\documentclass[10pt,a4paper,reqno]{amsart}
\usepackage{amsfonts,amsthm,latexsym,amsmath,amssymb,amscd,amsmath,
graphicx}

\newtheorem{tm}{Theorem}
\newtheorem{theo}[tm]{Theorem}

\newtheorem{lm}[tm]{Lemma}

\newtheorem{prop}[tm]{Proposition}
\theoremstyle{definition}

\newtheorem{conj}[tm]{Conjecture}
\newtheorem{ex}[tm]{Example}
\newcommand{\al}{\alpha}
\newcommand{\be}{\beta}
\newcommand{\ga}{\gamma}
\newcommand{\la}{\lambda}
\newcommand{\Si}{\Sigma}
\newcommand{\si}{\sigma}
\newcommand{\bC}{\mathbb C}

\newcommand{\bN}{\mathbb N}
\newcommand{\bZ}{\mathbb Z}

\newcommand{\C}{\mathcal C}
\newcommand {\A} {\mathcal A}

\newcommand \RR{\mathcal {R}}

\begin{document}
\title{Vandermonde varieties and relations among Schur polynomials} 

\author[R.~Fr\"oberg]{Ralf   Fr\"oberg}
\noindent 
\address{Department of Mathematics, Stockholm University, SE-106 91, Stockholm,  Sweden}
\email{ralff@math.su.se}

\author[B.~Shapiro]{Boris Shapiro}
\noindent 
\address{Department of Mathematics, Stockholm University, SE-106 91, Stockholm,  Sweden}
\email{shapiro@math.su.se}

\begin{abstract} Motivated by the famous Skolem-Mahler-Lech theorem we initiate in this paper  the study  of a natural class of determinantal  varieties  which we call {\em Vandermonde varieties}. They are closely related to the varieties consisting of  all linear recurrence relations of a given order possessing a non-trivial solution vanishing at a given set of integers.  In the regular case, i.e., when the dimension of a Vandermonde variety is the expected one,  we present its  free resolution,   obtain its degree and the Hilbert series. 
Some interesting relations among Schur polynomials are derived. Many open problems and conjectures are posed.
 \end{abstract}



\date{}
\maketitle 

\section{Introduction} 

The results in the present paper come from an attempt to understand the famous Skolem-Mahler-Lech theorem and its consequences.   Let us briefly recall its formulation. A linear recurrence relation with constant coefficients  of order $k$ is an equation   of the form
\begin{equation}\label{eq:Basic}
     u_{n}+\al_{1}u_{n-1}+\al_{2}u_{n-2}+\cdots+\al_{k}u_{n-k}=0,\; n\ge k
\end{equation}
where the  coefficients $(\al_1,...,\al_k)$ are fixed complex numbers and $\al_k\neq 0$. (Equation \eqref{eq:Basic} is  often referred to as a linear homogeneous difference equation with constant coefficients.) 

    The left-hand side of the equation
\begin{equation}\label{eq:Char}
	t^k+\al_{1}t^{k-1}+\al_{2}t^{k-2}+\cdots +\al_{k}=0
\end{equation}
is called the {\em characteristic polynomial} of
recurrence~\eqref{eq:Basic}. Denote the roots of~\eqref{eq:Char} (listed with possible repetitions) by
$x_{1},\ldots, x_{k}$ and call them the {\em characteristic roots} of \eqref{eq:Basic}. 

Notice that all $x_i$ are non-vanishing since $\al_k\neq 0$. 
 To obtain a concrete solution of \eqref{eq:Basic} one has  to prescribe additionally an initial $k$-tuple, $(u_0,\dots, u_{k-1})$, which can be chosen arbitrarily. Then $u_n,\;n\ge k$ are  obtained by using  the relation \eqref{eq:Basic}.  A solution of  \eqref{eq:Basic} is called {\em non-trivial} if not all of its entries vanish. In case of all distinct characteristic roots a general solution of \eqref{eq:Basic} can be given by 
$$u_n=c_1x_1^n+c_2x_2^n+...+c_kx_k^n$$ where $c_1,...,c_k$ are arbitrary complex numbers. In the general case of multiple characteristic roots a similar formula can be found in e.g. \cite {St}. 

An  arbitrary solution of a linear homogeneous difference (or differential) equation with constant coefficients of order $k$  is called an {\em exponential polynomial of order $k$}. One usually substitutes $x_i\neq 0$ by  $e^{\ga_i}$ and considers the obtained function in $\bC$ instead of $\bZ$ or $\bN$. (Other terms used for exponential polynomials are  {\em quasipolynomials} or {\em exponential sums}.) 

The most fundamental fact about  the structure of integer zeros of exponential polynomials is the well-known Skolem-Mahler-Lech theorem formulated below. It was first proved for recurrence sequences
of algebraic numbers  by K.~Mahler~\cite{Ma} in the 30's, based upon an idea of T.~Skolem~\cite{Sk}. Then, C.~Lech~\cite{Le}  published the result for general recurrence
sequences in 1953. In 1956 Mahler published the same result, apparently independently (but later realized to his chagrin that he had actually reviewed Lech's paper some years earlier, but had forgotten it).

\begin{theo} [The Skolem-Mahler-Lech theorem]  If $a_0,a _1, . . .$ is a solution to a linear recurrence relation, then the set of all k such that
$a_k = 0$ is the union of a finite (possibly empty) set and a finite number (possibly zero) of full arithmetic progressions.
(Here, a full arithmetic progression means a set of the form ${r, r + d, r + 2d, . . . }$  with $0 < r < d.$) 
\end{theo} 

A simple criterion guaranteeing the absence of arithmetic progressions is that no quotient of  two distinct characteristic roots  of the recurrence relation under consideration is a root of unity, see e.g. \cite{LoPo}. A recurrence relation \eqref{eq:Basic} satisfying this condition is called {\em  non-degenerate}. Substantial literature is devoted to finding the upper/lower bounds for the maximal number of arithmetic progressions/exceptional roots among all/non-degenerate linear recurrences of a given order. We give more details in \S~3. Our study is directly inspired by these investigations. 

\medskip
Let $L_k$ be the space of all linear recurrence relations \eqref{eq:Basic} of  order at most  $k$  with constant coefficients   and denote by $L_k^*=L_k\setminus \{\al_k=0\}$ the subset of all linear recurrence of order exactly $k$.  ($L_k$ is   the affine space with coordinates $(\al_1,...,\al_k)$.) To an arbitrary pair 
$(k;I)$ where $k\ge 2$ is a positive integer and $I=\{i_0<i_1<i_2<...<i_{m-1}\},\, m\ge k$ is a sequence of integers, we associate the variety $V_{k;I}\subset L_k^* $, the set of all linear recurrences of order  exactly $k$ having a non-trivial solution vanishing at all points of $I$. Denote by $\overline {V_{k;I}}$ the (set-theoretic) closure of $V_{k;I}$ in $L_k$. 
We call $V_{k;I}$ (resp. $\overline V_{k;I}$)  the open (resp. closed) {\em  linear recurrence variety} associated to  the pair $(k;I)$.  

In what follows we will always assume that $\gcd(i_1-i_0,...,i_{m-1}-i_0)=1$ to avoid unnecessary freedom related to the time rescaling in \eqref{eq:Basic}. 
Notice that since for  $m\le k-1$ one has $V_{k;I}=L_k^*$ and $\overline V_{k;I}= L_k$, this case does not require special consideration. A more important observation is  that due to translation invariance of \eqref{eq:Basic} for any integer $l$ and any pair $(k;I)$ the variety $V_{k;I}$ (resp. $\overline V_{k;I}$) coincides with the variety $V_{k;I+l}$ (resp. $\overline V_{k;I+l}$)  where the set of integers $I+l$ is obtained by adding $l$ to all entries of $I$.
 
So far we defined $\overline  V_{k;I}$ and $V_{k;I}$ as sets. However  for any pair $(k;I)$ the set $\overline V_{k;I}$  is an affine algebraic variety, see Proposition~\ref{pr:alg}. 
Notice that this fact is not completely obvious since if we, for example, instead of a set of integers choose as $I$ an arbitrary subset of real or complex numbers then the similar subset of $L_n$ will, in general,  only be  analytic.  

\medskip
Now we define the Vandermonde variety associated with a given pair $(k;I),\;  I=\{0\le i_0<i_1<i_2<...<i_{m-1}\},\; m\ge k$. Firstly, consider  the set $M_{k;I}$ of (generalized)  Vandermonde matrices  of the form 
\begin{equation}\label{eq:VdM}
M_{k;I}=\begin{pmatrix}  x_1^{i_0}&x_2^{i_0}&\cdots &x_k^{i_0}\\x_1^{i_1}&x_2^{i_1}&\cdots&x_k^{i_1}\\ \cdots&\cdots&\cdots&\cdots\\x_1^{i_{m-1}}&x_2^{i_{m-1}}&\cdots&x_k^{i_{m-1}}  \end{pmatrix},
\end{equation}
where $(x_1,...,x_k)\in \bC^k$.  In other words,   for a given  pair $(k;I)$ we take the map $M_{k;I}: \bC^k\to Mat(m,k)$ given by  \eqref{eq:VdM}ÔøΩ
where $Mat(m,k)$ is the space of all $m\times k$-matrices with complex entries and $(x_1,...,x_k)$ are chosen coordinates in $\bC^k$. 

We now define three slightly different but closely related versions of this variety as follows. 

\medskip\noindent
{\em Version 1.} Given a pair $(k;I)$ with $|I|\ge k$ define the {\em coarse Vandermonde variety} $Vd_{k;I}^{\bf c}\subset M_{k;I}$ as the set 
of all degenerate Vandermonde matrices, i.e., whose rank is smaller than $k$. $Vd_{k;I}^{\bf c}$  is obviously an algebraic variety whose defining ideal $\mathcal I_I$ is generated by all  $\binom {m}{k}$ maximal minors of $M_{k;I}$. Denote the quotient ring by $\RR_I=\RR/\mathcal I_I$. 

\medskip
Denote by $\A_k\subset \bC^k$ the standard Coxeter arrangement (of the Coxeter  group $A_{k-1}$)  consisting of all diagonals $x_i=x_j$ and by $\mathcal {BC}_k\subset \bC^k$ the Coxeter arrangement consisting of all $x_i=x_j$ and $x_i=0$. Obviously, $\mathcal BC_k\supset \A_k$. Notice that $Vd_{k;I}^{\bf c}$ always includes the arrangement $\mathcal {BC}_k$ (some of the hyperplanes with multiplicities) which is often inconvenient. Namely, 
with very few exceptions this means that $Vd_{k;I}^{\bf c}$ is not equidimensional, not CM, not reduced etc. For applications to  linear recurrences as well as questions in combinatorics and geometry of Schur polynomials it seems more natural to consider the localizations of $Vd_{k;I}^{\bf c}$ in $\bC^k\setminus \A_k$ and in $\bC^k\setminus \mathcal{BC}_k$. 

  \medskip
\noindent
{\em  Version 2.}
Define the $\A_k$-localization $Vd_{k;I}^{\A}$ of  $Vd_{k;I}^{\bf c}$ as the contraction of $Vd_{k;I}^{\bf c}$ to $\bC^k\setminus \A_k$.   Its is easy to obtain the  generating ideal of $Vd_{k;I}^\A$. Namely,  recall that given a sequence   ${ J}=(j_1<j_2<\cdots<j_k)$ of nonnegative integers one defines the associated Schur polynomial $S_{ J}(x_1,...,x_k)$ as given by 
 $$S_{ J}(x_1,\ldots,x_k)=\begin{vmatrix} x_1^{j_1}&x_2^{j_1}&\cdots&x_k^{j_2}\\ x_1^{j_2}&x_2^{j_2}&\cdots&x_k^{j_2}\\ \cdots&\cdots&\cdots&\cdots\\x_1^{j_k}&x_2^{j_k}&\cdots&x_k^{j_k}
 
 \end{vmatrix}/W(x_1,\ldots,x_k),
 $$ where $W(x_1,\ldots,x_k)$ is the usual Vandermonde determinant.  
 Given a sequence $I=(0\le i_0<i_1<i_2<\cdots<i_{m-1})$ with $\gcd(i_1-i_0,\ldots,i_{m-1}-i_0)=1$ consider the set of all  its $\binom m k$ subsequences ${ J}_\kappa$ of length $k$.   Here the index  $\kappa$ runs over the set of all subsequences of length $k$ among $\{1,2,..., m\}$. Take the corresponding 
 Schur polynomials $S_{ J_\kappa}(x_1,\ldots,x_k)$ and form the ideal $\mathcal I_{I}^\A$ in the polynomial ring $\bC[x_1,\ldots,x_k]$   generated by  all $\binom m k$ such Schur polynomials $S_{ J_\kappa}(x_1,\ldots,x_k)$. One can  show that the  Vandermonde variety $Vd_{k;I}^\A\subset \bC^k$ is generated by  $\mathcal I_{I}^\A,$ see Lemma~\ref{lm:deter}.   Denote the quotient ring by $\RR_I^\A=\RR/\mathcal I_I^\A$ where $\RR=\bC[x_1,...,x_k]$. Analogously, to the coarse Vandermonde variety $Vd_{k;I}$ the variety $Vd_{k;I}^\A$  often contains irrelevant coordinate hyperplanes which prevents it from having nice algebraic properties. For example, if $i_0>0$ then all coordinate hyperplanes necessarily belong to $Vd_{k;I}^\A$ ruining equidimensionality etc. On the other hand, under the assumption that $i_0=0$ the variety $Vd_{k;I}^\A$  often has quite reasonable properties presented below.
 
   \medskip
\noindent
{\em Version 3.}
Define the $\mathcal{BC}_k$-localization $Vd_{k;I}^{{BC}}$ of  $Vd_{k;I}^{\bf c}$ as the contraction of $Vd_{k;I}^{\bf c}$ to $\bC^k\setminus \mathcal {BC}_k$.   Again it is straightforward  to find the generating ideal of $Vd_{k;I}^{BC}$. Namely, 
  given a sequence   ${ J}=(0\le j_1<j_2<\cdots<j_k)$ of nonnegative integers define the reduced Schur polynomial $\hat S_{ J}(x_1,...,x_k)$ as given by 
 $$\hat S_{ J}(x_1,\ldots,x_k)=\begin{vmatrix} 1&1&\cdots&1\\ x_1^{j_2-j_1}&x_2^{j_2-j_1}&\cdots&x_k^{j_2-j_1}\\ \cdots&\cdots&\cdots&\cdots\\x_1^{j_k-j_1}&x_2^{j_k-j_1}&\cdots&x_k^{j_k-j_1}
 
 \end{vmatrix}/W(x_1,\ldots,x_k).
 $$
 In other words, $\hat S_J(x_1,...,x_k)$ is the usual Schur polynomial corresponding to  the sequence $(0,j_2-j_1,....,j_k-j_1)$. 
 Given a sequence $I=(0\le i_0<i_1<i_2<\cdots<i_{m-1})$ with $\gcd(i_1-i_0,\ldots,i_{m-1}-i_0)=1$ consider as before  the set of all  its $\binom m k$ subsequences ${ J}_\kappa$ of length $k$  where the index  $\kappa$ runs over the set of all subsequences of length $k$. Take the corresponding reduced 
 Schur polynomials $\hat S_{ J_\kappa}(x_1,\ldots,x_k)$ and form the ideal $\mathcal I_{I}^{BC}$ in the polynomial ring $\bC[x_1,\ldots,x_k]$   generated by  all $\binom m k$ such Schur polynomials $\hat S_{ J_\kappa}(x_1,\ldots,x_k)$. One can easily see that the  Vandermonde variety $Vd_{k;I}^{BC}\subset \bC^k$ as generated by  $\mathcal I_{I}^{BC}.$  Denote the quotient ring by $\RR_I^{BC}=\RR/\mathcal I_I^{BC}$. 
 
 \begin{conj}\label{conj:rad}
 If $\dim (Vd_{k;I}^{BC})\ge 2$ then  $\mathcal I_{I}^{BC}$ is a radical ideal. 
 \end{conj}

Notice that considered as sets the restrictions to $\bC^k\setminus \mathcal{BC}_k$  of all three varieties $Vd_{k;I}^{\bf c}$, $Vd_{k;I}^{\A}$, $Vd_{k;I}^{BC}$  coincide with  what we call  the {\em open Vandermonde variety}  $Vd^{op}_{k;I}$ which is  the subset of all matrices  of the form  $M_{k;I}$ with three properties:\\ 
(i) rank is smaller than $k$;\\ 
(ii) all $x_i$'s are non-vanishing;\\ 
(iii) all $x_i$'s are pairwise distinct.\\

\smallskip
Thus set-theoretically all the  differences between the three Vandermonde varieties are concentrated on the hyperplane arrangement $\mathcal{BC}_k$. Also from the above definitions it is obvious that $Vd^{op}_{k;I}$ and $Vd_{k;I}^{BC}$ are invariant under addition of an arbitrary integer  to $I$. 
The relation between the linear recurrence variety $V_{k;I}$ and the open Vandermonde variety $Vd_{k;I}^{op}$  is quite straight-forward. Namely,  consider the standard Vieta map:
\begin{equation}\label{Vieta}
Vi : \bC^k\to L_k
\end{equation}
sending an arbitrary  $k$-tuple $(x_1,...,x_k)$ to the polynomial $t^k+\al_1t^{k-1}+\al_2t^{k-2}+\cdots+\al_k$ whose roots are   $x_1,....,x_k$. Inverse images of the Vieta map are exactly the orbits of the standard $S_k$-action on $\bC^k$ by permutations of coordinates. Thus, the Vieta map sends a homogeneous and symmetric polynomial to a weighted homogeneous polynomial.

Define  the {\em open linear recurrence variety}  $V^{op}_{k;I}\subseteq V_{k;I}$ of a pair $(k;I)$ as consisting of all recurrences in $V_{k;I}$  with all distinct characteristic roots distinct. The following  statement is obvious. 

\begin{lm}\label{lm:push} The map $Vi$ restricted to $Vd^{op}_{k;I}$ gives an unramified $k!$-covering of  the set $V^{op}_{k;I}$. 
\end{lm}

  Unfortunately at the present moment the following natural question is still open. 

\medskip
\noindent
{\em Problem 1.} Is it true that  for any pair $(k;I)$ one has that $\overline{V^{op}_{k;I}}=V_{k;I}$ where  $\overline{V^{op}_{k;I}}$ is the set-theoretic closure of  ${V^{op}_{k;I}}$ in $L_k^*$? If  'not', then under what additional assumptions?

\medskip
Our main results are as follows. Using the Eagon-Northcott resolution of determinantal ideals, we determine the resolution, and hence the Hilbert series and degree
of $\RR_I^\A$ in Theorem~\ref{th:general}. We give an alternative calculation of the degree using the Giambelli-Thom-Porteous formula in Proposition~\ref{pr:kaz}.
In the simplest non-trivial case, when $m=k+1$, we get more detailed information about $Vd_{k;I}^{\A}$. We prove that its codimension is 2, and that $\RR_I^\A$
is Cohen-Macaulay. We also discuss minimal sets of generators of $\mathcal I_I$, and determine when we have a complete intersection in Theorem~\ref{th:k*k+1}.
(The proof of this theorem gives a lot of interesting relations between Schur polynomials, see Theorem~\ref{Arel}.) 
In this case the variety has the expected codimension, which is not always the case if $m>k+1$. In fact  our computer experiments suggest that then the codimension
rather seldom is the expected one. In case $k=3$, $m=5$, we show that having the expected codimension is equivalent to $\RR_I^\A$ being a complete intersection and
that $\mathcal I_I$ is generated by three complete symmetric functions.  Exactly the  problem (along with many other similar questions)  when three complete symmetric functions constitute a regular sequence was considered in  an recent paper \cite{CKW} where the authors formulated a detailed conjecture. We slightly strengthen their conjecture below.

\smallskip
For the $\mathcal{BC}_k$-localized variety $Vd_{k;I}^{BC}$ we have only proofs when $k=3$, but we present Conjectures~\ref{th:k*k+1??} and~\ref{Brel}, supported
by many calculations.  We end the paper with a section which describes the connection of our work with the fundamental problems in linear recurrence
relations.

\medskip 
\noindent 
{\em Acknowledgements.} The authors want to thank Professor Maxim Kazarian (Steklov Institute of Mathematical Sciences) for his help with Giambelli-Thom-Porteous formula, Professor Igor Shparlinski (Macquarie University)  
for highly  relevant  information on the Skolem-Mahler-Lech theorem and Professors Nicolai Vorobjov (University of Bath)  and Michael Shapiro (Michigan State University)  for discussions. We are especially grateful to Professor Winfried Bruns (University of Osnabr\"uck) for pointing out important information on determinantal ideals. 

\section{Results and conjectures on Vandermonde varieties}
We start by proving that $\overline V_{k;I}$  is an affine algebraic variety, see Introduction.

\begin{prop}\label{pr:alg}
For any pair $(k;I)$ the set $\overline V_{k;I}$  is an affine algebraic variety. Therefore, $V_{k;I}=\overline V_{k;I}\vert_{L^*_k}$
 is a quasi-affine variety. \end{prop}

\begin{proof} We will show that for any pair $(k;I)$ the variety $\overline V_{k;I}$ of linear recurrences is  constructible. Since it is by definition closed in the usual topology of $L_k\simeq \bC^k$  it is  algebraic. The latter fact follows  from  \cite{Mu}, I.10 Corollary 1 claiming that 
if $Z \subset X$ is a constructible subset of a variety, then the Zarisky
closure and the strong closure of $Z$ are the same. Instead of showing that $\overline V_{k;I}$ is constructible we prove that $V_{k;I}\subset L_k^*$ is constructible. Namely, we can use an analog of Lemma~\ref{lm:push} to construct a natural stratification of $V_{k;I}$ into the images of quasi-affine sets under appropriate Vieta maps.  Namely, let us stratify $V_{k;I}$ as 
$V_{k;I}=\bigcup_{\la \vdash k} V_{k;I}^\la$ where $\la \vdash k$ is an arbitrary  partition of $k$ and $V_{k;I}^\la$  is the subset 
 of $V_{k;I}$ consisting of all recurrence relations of length exactly $k$ which has a non-trivial solution vanishing at each point of $I$ and whose characteristic polynomial determines the partition $\la$ of its degree $k$. In other words, if $\la=(\la_1,...,\la_s),\; \sum_{j=1}^s\la_j=k$ then the characteristic polynomial should have $s$ distinct roots of multiplicities $\la_1,...,\la_s$ resp.  Notice that any of these $V_{k;I}^\la$ can be empty including the whole $V_{k;I}$ in which case there is nothing to prove. Let us now show that each $V_{k;I}^\la$ is the image  under  the appropriate Vieta map of  a set similar to the open Vandermonde variety. Recall that if $\la=(\la_1,...,\la_s),\; \sum_{j=1}^s \la_j=k$ and $x_1,..,x_s$ are the distinct roots with the multiplicities $\la_1,...,\la_s$ respectively of the linear recurrence~\eqref{eq:Basic} then the general solution of \eqref{eq:Basic} has the form 
 $$u_{n}=P_{\la_1}(n)x_1^{n}+P_{\la_2}(n)x_2^n+...+P_{\la_s}(n)x_s^n,$$
  where $P_{\la_1}(n),..., P_{\la_s}(n)$ are arbitrary polynomials in the variable $n$ of degrees $\la_1-1, \la_2-1,....,\la_{s-1}$ resp. 
 Now, for a given $\la\vdash k$ consider the set of matrices 
 $$M^\la_{k;I}=$$
 $$\begin{pmatrix} x_1^{i_0}&i_0x_1^{i_0}&...& i_0^{\la_1-1}x_1^{i_0}&...& x_s^{i_0}&i_0x_s^{i_0}&...& i_0^{\la_s-1}x_s^{i_0}\\ 
 x_1^{i_1}&i_1x_1^{i_1}&...& i_1^{\la_1-1}x_1^{i_1}&...& x_s^{i_1}&i_1x_s^{i_1}&...& i_1^{\la_s-1}x_s^{i_1}\\ 
 \cdots&\cdots&\cdots&\cdots& \cdots&\cdots&\cdots&\cdots&\cdots\\
 x_1^{i_{m-1}}&i_{m-1}x_1^{i_{m-1}}&...& i_{m-1}^{\la_1-1}x_1^{i_{m-1}}&...& x_s^{i_{m-1}}&i_1x_s^{i_{m-1}}&...& i_{m-1}^{\la_s-1}x_s^{i_{m-1}}
 \end{pmatrix}.$$
In other words,  we are taking the fundamental solution $x_1^n,nx_1^n,...,n^{\la_1-1}x_1^n,$ $ x_2^n,$ $ nx_2^n,$ $...,n^{\la_2-1}x_2^n,....,$ $ x_s^n, nx_s^n,...,n^{\la_s-1}x_s^n$ of \eqref{eq:Basic} under the assumption that the characteristic polynomial gives a partition $\la$ of $k$  and we are evaluating each function in this system at  $i_0, i_1,..., i_{m-1}$ resp. We now define the variety $Vd^\la_{k;I}$ as the subset of matrices of the form $M^\la_{k;I}$ such that:  (i) the rank of such a matrix is smaller than $k$; (ii) all $x_i$ are distinct; (iii) all $x_i$ are non-vanishing. Obviously, $Vd^\la_{k;I}$ is a  quasi-projective variety in $\bC^s$. Define  the analog $Vi_\la: \bC^s\to L_k$ which sends an $s$-tuple $(x_1,...,x_s)\in \bC^s$ to the polynomials $\prod_{j=1}^s(x-x_j)^{\la_j}\in L_k$ of the Vieta map $Vi$.  One can easily see that 
 $Vi_\la$ maps $Vd^\la_{k;I}$ onto $V_{k;I}^\la$. Applying this construction to all partitions $\la\vdash k$ we will obtain that 
 $V_{k;I}=\bigcup_{\la \vdash k} V_{k;I}^\la$  is constructible which finishes the proof.   
\end{proof}

The remaining part of the paper is devoted to the study of the Vandermonde varieties $Vd_{k;I}^\A$ and $Vd_{k;I}^{BC}$. We start with the $\A_k$-localized variety $Vd_{k;I}^\A$. Notice that if $m=k$ the variety $Vd_{k;I}^\A\subset \bC^k$ is an irreducible hypersurface given by  the equation $S_I=0$ and its degree equals $\sum_{j=0}^{k-1}i_j-\binom{k}{2}$. We will need the following alternative description of the ideal $\mathcal I_I^\A$ in the general case. Namely, using the Jacobi-Trudi identity for  the Schur polynomials we get the following statement.

\begin{lm}\label{lm:deter}
For any pair $(k;I),\;I=\{i_0<i_1<...<i_{m-1}\}$ the  ideal $\mathcal I_I^\A$ is generated by all $k\times k$-minors of the $m\times k$-matrix

\begin{equation}\label{Repr} 
H_{k;I}=\left(\begin{array}{cccc}
h_{i_0-(k-1)}&h_{i_0-(k-2)}&\cdots&h_{i_0}\\
h_{i_{1}-(k-1)}&h_{i_{1}-(k-2)}&\cdots&h_{i_{1}}\\
\vdots&\vdots&\vdots&\vdots\\
h_{i_{m-1}-(k-1)}&h_{i_{m-1}-(k-2)}&\cdots&h_{i_{m-1}}\\
\end{array}\right).
\end{equation}
Here $h_i$ denotes the complete symmetric function of degree $i$,
 $h_i=0$ if $i<0$, $h_0=1$.
\end{lm}

\begin{proof}
 It follows directly from the standard Jacobi-Trudi identity for the Schur polynomials, see e.g. \cite {Tam}. 
\end{proof}

In particular, Lemma~\ref{lm:deter} shows that $Vd_{k;I}^\A$ is an determinantal variety in the usual sense. When working with $Vd_{k;I}^\A$ and unless the opposite is explicitly mentioned  we will  assume that $I=\{0<i_1<...<i_{m-1}\}$, i.e. that $i_0=0$ and that additionally $\gcd(i_1,...,i_{m-1})=1$. Let us first study  some properties of $Vd_{k;I}^\A$ in the so-called {\em regular case,} i.e.  when its dimension coincides with the expected one. 

Namely, consider  the set  $\Omega_{m,k}\subset Mat(m,k)$ of all $m\times k$-matrices having positive corank. It is well-known that   $\Omega_{m,k}$ has codimension equal to $m-k+1$. Since $Vd_{k;I}^{\bf c}$ coincides with the pullback of $\Omega_{m,k}$ under the map $M_{k;I}$ and  $Vd_{k;I}^{\A}$ is closely related to it (but with trivial pathology on $\A_k$ removed)  the expected codimension of 
$Vd_{k;I}^\A$ equals $m-k+1$. We call a pair $(k;I)$ {\em $\A$-regular} if  $k\le m\le 2k-1$ (implying that the expected dimension of $Vd_{k;I}^\A$ is positive) and the actual codimension of $Vd_{k;I}^\A$ coincides with its expected codimension.  We now describe the Hilbert series of the quotient ring $\RR_I^\A$ in the case of a arbitrary  regular pair $(k;I)$ using  the well-known resolution of determinantal ideals of Eagon-Northcott \cite{EN}. 

To explain the notation in the following theorem, we introduce two
gradings, $\rm{tdeg}$ and $\deg$, on $\bC[t_0,\ldots,t_{m-1}]$. The first one is  the usual
grading induced by ${\rm tdeg}(t_i)=1$ for all $i$, and a second one is  induced
by $\deg(t_i)=-i$. In the next theorem $M$ denotes a monomial in
$\bC[t_0\ldots,t_{m-1}]$.

\begin{theo}\label{th:general} 
In the above notation, and with $I'=\{ i_1,\ldots,i_{m-1}\}$, \\
\begin{itemize}
\item[(a)]  the Hilbert series $Hilb_I^\A(t)$ of $\RR_I^\A=\RR/\mathcal I_I^\A$ is given by
$$Hilb_I^\A(t)=\frac{1-\sum_{i=0}^{m-k}((-1)^{i+1}\sum_{J\subseteq I',|J|=k+i}t^{s_J}\sum_{M\in N_i}t^{\deg(M)})}{(1-t)^k},$$
where $s_J=\sum_{j\in J}j-{k\choose2}$ and $N_i=\{ M;{\rm tdeg}(M)=i\}$.\\

\item[(b)] The degree of $\RR_I^\A$ is $T^{(m-k+1)}(1)(-1)^{m-k+1}/(m-k+1)!$, where $T(t)$ is the numerator in \rm(a).
\end{itemize}
\end{theo}
\begin{proof}  According to \cite {EN}  provided that $\mathcal I_I^\A$  has the expected codimension $m-k+1$, it is known to be Cohen-Macaulay and it has a resolution  of the form
\begin{equation}\label{resol}
0\rightarrow F_{m-k+1}\rightarrow\cdots\rightarrow F_1\rightarrow \RR\rightarrow \RR_I^\A\rightarrow 0,
\end{equation} 
where $F_j$ is free module  over $\RR=\bC[x_1,\ldots,x_k]$ of rank ${m\choose k+j-1}{k+j-2\choose k-1}$. We denote the basis elements of $F_j$ by $M_IT$, where $I\subseteq\{i_0,\ldots.i_{m-1}\}$,
$|I|=k+j-1$, and $T$ is an arbitrary monomial in $\{ t_0,\ldots,t_{k-1}\}$ of degree $j-1$. 
 If $M_I=\{ i_{l_1},\ldots,i_{l_{k+j-1}}\}$ and $T=t_{j_1}^{s_1}\cdots t_{j_r}^{s_r}$ with $s_i>0$ for all $i$ and $\sum_{i=1}^rs_i=j-1$, then, in our situation, $d(M_IT)=\sum_{i=1}^r(\sum_{l=1}^{k+j-1}(-1)^{k+1}h_{i_{k_l}-j_i}M_{I\setminus\{ i_{k_l}\}})T/t_{j_i}$. Here $\deg(M_I)=\sum_{n=1}^{k+j-1}i_{l_j}-{k\choose2}$ and $\deg(t_i)=-i$. (Note that tdeg$(t_i)=1$ but $\deg(t_i)=-i$.) Thus $\deg(M_IT)=\sum_{n=1}^{k+j-1}i_{l_n}+\sum_{i=1}^rs_ij_i-{k\choose2}$ if $M_I=\{ i_{l_1},\ldots,i_{l_{k+j-1}}\}$ and $T=t_{j_1}^{s_1}\cdots t_{j_r}^{s_r}$.  Observe that this resolution is never minimal. Indeed, for any sequence $I=\{0=i_0<i_1\cdots<i_{m-1}\}$, we only need the Schur polynomials coming from subsequences starting with 0, so $\mathcal I_I^\A$ is generated by at most ${m-1\choose k-1}$ Schur polynomials instead of totally ${m\choose k}$; see also discussions preceding the proof of Theorem~\ref{th:k*k+1} below.  Now, if $\mathcal J$ is an arbitrary homogeneous ideal in $\RR=\bC[x_1,\ldots,x_k]$ and $\RR/\mathcal J$ has a resolution $$0\rightarrow\oplus_{i=1}^{\beta_r}\RR(-n_{r,i})\rightarrow\cdots\rightarrow\oplus_{i=1}^{\beta_1}\RR(-n_{1.i})\rightarrow \RR\rightarrow \RR/\mathcal J\rightarrow 0,$$ then  
the Hilbert series of $\RR/\mathcal J$ is given by $$\frac{1-\sum_{i=1}^{\beta_1}t^{n_{1,i}}+\cdots+(-1)^r\sum_{i=1}^{\beta_r}t^{n_{r,i}}}{(1-t)^k}.$$
For the resolution \eqref{resol}, all terms coming from $M_IT$ with $t_0|T$ cancel. Thus we get the claimed Hilbert series. If the Hilbert series is given by $T(t)/(1-t)^k=P(t)/((1-t)^{\dim(\RR/\mathcal I_I)}$, then the degree of the corresponding variety equals $P(1)$. We have $T(t)=(1-t)^{m-k+1}P(t)$, or after differentiating  the latter identity $m-k+1$ times we get $P(1)=T^{(m-k+1)}(1)(-1)^{m-k+1}/(m-k+1)!$.
\end{proof} 

\begin{ex} For the case $3\times5$ with $I=\{0,i_1,i_2,i_3,i_4)$ , if the ideal $Vd^\A_{k;I}$ has the right codimension, we get that its Hilbert series equals $T(t)/(1-t)^3$, where
$$T(t)=1-t^{i_1+i_2-2}-t^{i_1+i_3-2}-t^{i_1+i_4-2}-t^{i_2+i_3-2}-t^{i_2+i_4-2}-t^{i_3+i_4-2}+t^{i_1+i_2+i_3-3}+$$
$$t^{i_1+i_2+i_4-3}+t^{i_1+i_3+i_4-3}+t^{i_2+i_3+i_4-3}+t^{i_1+i_2+i_3-4}+t^{i_1+i_2+i_4-4}+t^{i_1+i_3+i_4-4}+$$
$$t^{i_2+i_3+i_4-4}-t^{i_1+i_2+i_3+i_4-3}-t^{i_1+i_2+i_3+i_4-4}-t^{i_1+i_2+i_3+i_4-5}$$ and the degree of $Vd_{k;I}^\A$ equals  
$$i_1i_2i_3+i_1i_2i_4+i_1i_3i_4+i_2i_3i_4-3(i_1i_2+i_1i_3+i_1i_4+i_2i_3+i_2i_4+i_3i_4)+7(i_1+i_2+i_3+i_4)-15.$$
\end{ex}

An alternative way to calculate $\deg(Vd_{k;I}^\A)$ is to use the Giambelli-Thom-Porteous formula, see e.g. \cite {Fu}. 
The next result  corresponded to the authors by M.~Kazarian explains how to do that.  

\begin{prop}\label{pr:kaz}  Assume that $Vd_{k;I}^\A$ has the expected codimension $m-k+1$. Then  its degree (taking multiplicities of the components into
account) is equal to the coefficient of $t^{m-k+1}$ in the Taylor expansion of the series
$$\frac{\prod_{j=1}^{m-1}(1+i_jt)}{\prod_{j=1}^{k-1}(1+jt)}.$$
More explicitly,
$$\deg(Vd_{k;I}^\A) =\sum_{j}^{m-k+1}\si_j(I)u_{m-k+1-j}$$
where $\si_j$ is the $j$th elementary symmetric function of the entries $(i_1,...,i_{m-1})$ and $u_0, u_1, u_2,...$ are the coefficients in the Taylor expansion of $\prod_{j=1}^{k-1}\frac{1}{1+jt}$, i.e. $u_0+u_1t+u_2t^2+...=\prod_{j=1}^{k-1}\frac{1}{1+jt}$. In particular, $u_0=1$, $u_1=-\binom {k}{2},$ $u_2=\binom{k+1}{3}\frac{3k-2}{4},$ $u_3=-\binom{k+2}{4}\binom {k}{2},$ $u_4 =\binom {k + 3}{5}\frac{15 k^3-15 k^2-10 k+8}{48}$. 
\end{prop}  

\begin{proof} 
In the Giambelli formula setting, we consider a "generic" family of $(n\times l)$-matrices 
$A = || a_{p,q}||,\; 1 \le p \le n,\; 1 \le q \le l$, whose entries are homogeneous functions  
of degrees $\deg (a_{p,q}) = \al_p-\be_q$ in parameters $(x_1,...,x_k)$ for some fixed sequences  $\be = (\be_1,..., \be_l)$ and $\al =(\al_1,..., \al_n)$. Denote by $\Si^r$ the subvariety in the parameter space $\bC^k$ determined
by the condition that the matrix $A$ has rank at most $l-r$, that is, the linear
operator $A : \bC^l \to \bC^n$ has at least a $r$-dimensional kernel.
Then the expected codimension of the subvariety $\Si^r$  is equal to
$$\text{codim}(\Si^r)= r (n-l + r).$$
 In case when the actual codimension coincides with the expected one its degree is computed as the following $r\times r$-determinant: 
\begin{equation}\label{eq:Giam}
\deg (\Si^r)=\det ||c_{n-l+r-i+j}||_{1\le i,\;j\le r}, 
\end{equation}
where the entries $c_i$'s are defined by the Taylor expansion
$$1 + c_1t + c_2t^2 +...=\frac{\prod^n_{p=1}(1 + \al_pt)}{\prod^l_{q=1}(1 + \be_qt)}.$$ 
There is a number of situations where this formula can be applied.
Depending on the setting, the entries $\al_p, \be_q$ can be rational numbers, formal variables, first Chern classes of line bundles or formal Chern roots of vector bundles of ranks $n$ and $l$, respectively.
In the situation of Theorem~\ref{th:general} we should use the presentation \eqref{Repr} of $Vd_{k;I}^\A$ from  Lemma~\ref{lm:deter}. Then  we have $n=m,\; l=k,\; r = 1$,  $\al = I = (0, i_1,...,i_{m-1}),\;  \be = (k-1,k-2,...,0)$. Under the assumptions of Theorem~\ref{th:general} the  degree of the Vandermonde variety $Vd_{k;I}^\A$ will be given by  the $1\times 1$-determinant of the Giambelli-Thom-Porteous formula~\eqref{eq:Giam}, 
that is, the coefficient $c_{m-k+1}$ of $t^{m-k+1}$  in the expansion of
$$1 + c_1t + c_2t^2 +...=\frac{\prod^{m-1}_{j=0}(1 + i_jt)}{\prod_{j=1}^{k}(1 + (k-j)t)}=\frac{\prod^{m-1}_{j=1}(1 + i_jt)}{\prod_{j=1}^{k-1}(1 + jt)},$$
which gives exactly the stated formula for $\deg(Vd_{k;I}^\A)$. 
\end{proof}

In the simplest non-trivial case $m=k+1$ one can obtain more detailed information about $Vd_{k;I}^\A$. Notice that for $m=k+1$ the $k+1$ Schur polynomials generating the ideal $\mathcal{I}_I^\A$ are naturally ordered according to  their degree. Namely, given an arbitrary $I=\{0<i_1<i_1<...<i_k\}$ with $\gcd(i_1,...,i_k)=1$  denote by $S_j,\;j=0,...,k$ the Schur polynomial obtained by removal of the $(j)$-th row of the matrix $M_{k;I}$. (Pay attention that  here we enumerate the rows starting from $0$.)  Then, obviously, $\deg S_k<\deg S_{k-1} <....<\deg S_0$. Using presentation \eqref{Repr}  we get the following. 

\begin{tm}\label{th:k*k+1} For any integer sequence $I=\{0=i_0<i_1<i_2<...<i_k\}$ of length $k+1$ with $\gcd(i_1,...,i_k)=1$ the following facts are valid:  
\begin{itemize}

\item[(i)] ${\text codim} (Vd_{k;I}^\A)=2$;

\item[(ii)] The quotient ring $\RR_{I}^\A$ is Cohen-Macaulay;

\item[(iii)] The Hilbert series $Hilb_I^\A(t)$ of $\RR_{I}^\A$ is given by the formula
$$Hilb_I^\A(t)=\left(1-\sum_{j=1}^kt^{N-i_j-\binom{k}{2}}+\sum_{j=1}^{k-1}t^{N-j-\binom{k}{2}}\right)/(1-t)^k,$$
where $N=\sum_{j=1}^k i_j$; 

\item[(iv)] $\deg (Vd_{k;I}^\A)=\sum_{1\le j<l\le k}i_ji_l
-{k\choose2}\sum_{j=1}^ki_j+{k+1\choose3}(3k-2)/4$;

\item[(v)]  The ideal $\mathcal I_{I}^\A$  is always generated by  $k$ generators $S_k,..., S_{1} $ (i.e., the last generator $S_0$ always lies in the ideal generated by  $S_k,..., S_{1} $). Moreover, if for some $1\le n \le k-2$ one has $i_n\le k-n$ then  $\mathcal I_{I}^\A$ is generated by $k-n$ elements $S_k,...,S_{n+1}$. 
In particular, it is generated by two elements $S_k, S_{k-1}$ (i.e., is a complete intersection) if $i_{k-2}\le k-1$.

\end{itemize}
\end{tm}

The theorem gives lots of relations between Schur polynomials.
\begin{tm}\label{Arel}  
Let the generators be $S_k=s_{i_{k-1}-k+1,i_{k-2}-k+2,\ldots,i_1-1},S_{k-1},\ldots,S_0=s_{i_k-(k-1),i_{k-1}-(k-2),\ldots,i_1}$ in degree increasing order. For $s=0,1\ldots,k-1$ we have
$$h_{i_k-s}S_k-h_{i_{k-1}-s}S_{k-1}+\cdots+(-1)^{k-1}h_{i_1-s}S_1+(-1)^kh_{-s}S_0=0.$$
Here $h_i=0$ if $i<0$.
\end{tm}
To prove Theorems~\ref{th:k*k+1} and ~\ref{Arel} notice that since  Schur polynomials are irreducible \cite{DZ}, 
in the case
$m=k+1$ the ideal $\mathcal I_I^\A$ always has the  expected codimension 2 unless it coincides with the whole ring $\bC[x_1,...,x_k]$.  Therefore vanishing of any two Schur polynomials lowers the dimension by two. (Recall that we assume that $\gcd(i_1,...,i_k)=1$.) On the other hand, as we mentioned in the introduction the codimension of $Vd_{k;I}^\A$ in this case is at most $2$. For $m=k+1$ one can present a very concrete resolution of the quotient ring  $\RR_I^\A$.

Namely, given a  sequence $I=\{0=i_0<i_1<\cdots<i_k\}$ we know that  the ideal $\mathcal I_I^\A$ is generated by the  $k+1$ 
Schur polynomials
$S_{l}=s_{a_k,a_{k-1},\ldots,a_1},\;l=0,\ldots,k$, where 
$$(a_k,\ldots,a_1)=(i_k,i_{k-1},\ldots,i_{l+1},\hat{i_l},i_{l-1},\ldots,i_0)-
(k-1,k-2,\ldots,1,0).$$ 
Obviously, $S_{l}$ has degree
$\sum_{j=1}^ki_j-i_l-{k\choose2}$ and by the Jacobi-Trudi identity is given by 
\begin{displaymath}
S_{l}=\left|\begin{array}{cccc}
h_{i_0-(k-1)}&h_{i_0-(k-2)}&\cdots&h_{i_0}\\
h_{i_1-(k-1)}&h_{i_1-(k-2)}&\cdots&h_{i_1}\\
\vdots&\vdots&\vdots&\vdots\\
h_{i_{l-1}-(k-1)}&h_{i_{l-1}-(k-2)}&\cdots&h_{i_{l-1}}\\
h_{i_{l+1}-(k-1)}&h_{i_{l+1}-(k-2)}&\cdots&h_{i_{l+1}}\\
\vdots&\vdots&\vdots&\vdots\\
h_{i_{k-1}-(k-1)}&h_{i_{k-1}-(k-2}&\cdots&h_{i_{k-1}}\\
h_{i_k-(k-1)}&h_{i_k-(k-2)}&\cdots&h_{i_k}\\
\end{array}\right|.
\end{displaymath}
Here (as above) $h_j$ denotes the complete symmetric function of degree $j$ in $x_1,...,x_k$. (We set $h_j=0$ if $j<0$ and $h_0=1$.) 
  Consider the $(k+1)\times k$-matrix $H=H_{k;I}$ given by 
\begin{displaymath}
H=\left(\begin{array}{cccc}
h_{i_0-(k-1)}&h_{i_0-(k-2)}&\cdots&h_{i_0}\\
h_{i_1-(k-1)}&h_{i_1-(k-2)}&\cdots&h_{i_1}\\
\vdots&\vdots&\vdots&\vdots\\

h_{i_{k-1}-(k-1)}&h_{i_{k-1}-(k-2}&\cdots&h_{i_{k-1}}\\
h_{i_k-(k-1)}&h_{i_k-(k-2)}&\cdots&h_{i_k}\\

\end{array}\right).
\end{displaymath}
Let $H_l$ be the $(k+1)\times(k+1)$-matrix obtained  by
extending $H$ with $l$-th column  of $H$. Notice that  $\det(H_l)=0$, and expanding it along the last
column we get for $0\le l\le k-1$ the relation
$$0=\det(H_l)=h_{i_k-(k-l)}S_k-h_{i_{k-1}-(k-l)}S_{k-1}+\cdots+(-1)^{k-1}h_{i_1-(k-l)}S_1.$$
For $l=k$ we get
$$h_{i_k}S_k-h_{i_{k-i}}S_{k-1}+\cdots+(-1)^kh_{i_0}S_0=0$$
which implies that $S_0$ always lie in the ideal generated by the remaining $S_1,\ldots,S_k$. 

\medskip
We now prove Theorem~\ref{th:k*k+1}. 
\begin{proof} Set $N=\sum_{j=1}^ki_j$. 
 For an arbitrary $I=\{0, i_1,...,i_k)$ with $\gcd(i_1,...,i_k)=1$ we 
get the following resolution of the quotient ring $\RR_I^\A=\RR/\mathcal I_I^\A$  
$$0\longrightarrow \oplus_{l=1}^k\RR(-N+{k\choose2}+l)\longrightarrow
\oplus_{l=1}^{k-1}\RR(-N+i_l+{k\choose2})\longrightarrow \RR\longrightarrow \RR_I^\A
\longrightarrow 0$$ where $\RR=\bC[x_1,...,x_k]$. Simple calculation with this resolution  implies that  the Hilbert series $Hilb_I^\A(t)$ of $\RR_I^\A$ is given by 
$$Hilb_I^\A(t)=\left(1-\sum_{l=1}^kt^{N-i_l-{k\choose2}}+\sum_{l=1}^{k-1}t^{N-{k\choose2}-l}\right)/(1-t)^k$$
and the degree of $Vd_{k;I}^\A$ is given by 
$$\deg(Vd_{k;I}^\A)=\sum_{1\le r<s\le k}i_ri_s-{k\choose2}\sum_{r=1}^ki_r+{k+1\choose3}(3k-2)/4.$$

Notice that  the latter resolution might not be minimal, since the ideal might have fewer than $k$ generators.
To finish proving  Theorem~\ref{th:k*k+1} notice that if conditions of (v) are satisfied  then a closer look at the resolution reveals that the Schur polynomials $S_0,\ldots,S_{k-n}$ lie in the ideal generated by $S_{k-n+1},\ldots,S_k$.
\end{proof}

In connection with Theorems~\ref{th:general} and ~\ref{th:k*k+1} the following question is completely natural. 

\medskip
\noindent
{\em Problem 2.}   Under the assumptions $i_0=0$ and $\gcd(i_1,...,i_{m-1})=1$ which pairs $(k;I)$ are $\A$-regular? 
\medskip 

Theorem~\ref{th:k*k+1} shows that for $m=k+1$ the condition $\gcd(i_1,...,i_k)=1$ guarantees regularity of any pair $(k;I)$ with $|I|=k+1$. On the other hand, our computer experiments with Macaulay  suggest that for $m>k$  regular cases  are rather seldom. 
In particular, we were able to prove the following. 

\begin{theo}\label{th:expected} If $m>k$ a necessary  (but insufficient)  condition for $Vd_{k;I}^\A$  to have the expected codimension is $i_1=1$.
\end{theo}

\begin{proof}If $i_1\ge2$, then $i_{k-2}\ge k-1$. This means that the ideal is generated by Schur polynomials $s_{a_0,\ldots,a_{k-1}}$ with $a_{k-2}\ge1$. Multiplying these up to degree $n$
gives linear combinations of Schur polynomials $s_{b_1,\ldots,b_{k-1}}$ with $b_{k-2}\ge1$.
Thus we miss all Schur polynomials with $b_{k-2}=0$. The number of such Schur polynomials equals the number of partitions of $n$ in at most $k-2$ parts. The number of partitions of $n$ in exactly $k-2$
parts is approximated with $n^{k-3}/((k-2)!(k-1)!)$. Thus the number of elements of degree $n$ in the ring is at least $cn^{k-3}$ for some positive $c$, so the ring has dimension $\ge k-2$. The expected dimension is $\le k-3$, which is a contradiction.  
\end{proof}

So far a complete (conjectural) answer to Problem 2 is only  available in the first non-trivial case $k=3,\; m=5$. Namely, for a $5$-tuple $I=\{0,1,i_2,i_3,i_4\}$ to be regular one needs  the corresponding the Vandermonde variety $Vd_{3;I}^\A$ to be a complete intersection. This is due to the fact that in this situation the ideal  $\mathcal I_I^\A$ is generated by the Schur polynomials $S_4,S_3,S_2$ of the least degrees in the above notation.  Notice that $S_4=h_{i_2-2},\; S_3=h_{i_3-2},\; S_2=h_{i_4-2}$. Thus $Vd_{3;I}^\A$ has the expected codimension (equal to $3$)  if and only if $\bC[x_1,x_2,x_3]/\langle h_{i_2-2},h_{i_3-2},h_{i_4-2}\rangle$
is a complete intersection or, in other words, $h_{i_2-2},h_{i_3-2},h_{i_4-2}$ is a regular sequence. Exactly this  problem (along with many other similar questions)  was considered in   intriguing paper \cite{CKW} where the authors formulated the following claim, see Conjecture 2.17 of \cite{CKW}. 

\begin{conj}\label{conj:CKW} Let $A=\{a,b,c\}$ with $a<b<c$. Then $h_a, h_b, h_c$ in three variables is a regular sequence if and only if the following conditions are satisfied: 
\begin{itemize}
\item[(1)] $abc\equiv 0  \mod 6;$
\item[(2)] $\gcd(a+1,b+1,c+1)=1;$
\item[(3)] For all $t\in \mathbb N$ with $t>2$ there exists $d\in A$ such that $d+2\not \equiv 0,1\mod t$.

\end{itemize}

\end{conj} 

In fact, our experiments allow us to strengthen the latter conjecture in the following way. 

\begin{conj}\label{conj:CKWst} In the above set-up if the sequence $h_a, h_b, h_c$ with $a>1$ in three variables is not regular, then $h_c$ lies in the ideal 
generated by $h_a$ and $h_b$.  (If $(a,b,c)=(1,4,3k+2)$, $k\ge1$, then $h_a,h_b,h_c$ neither is a regular sequence, nor $h_c\in(h_a,h_b)$.) 
\end{conj}

We note that if we extend the set-up of \cite {CKW} by allowing Schur polynomials  $s(r,s,t)$  instead of just complete symmetric functions  then if $ t>0$ in all three of them the sequence is never regular.  Conjectures~\ref{conj:CKW} and ~\ref{conj:CKWst} provide a criterion which agrees with our calculations of $\dim(Vd_{3;I}^\A)$. 
Finally, we made experiments checking how   $\dim(Vd_{k;I}^\A)$ depends on the last entry $i_{m-1}$ of $I=\{0,1,i_2,...,i_{m-1}\}$ while keeping the first $m-1$ entries fixed. 

\begin{conj}\label{conj:period}
For any given $I=(0,1,i_2,...,i_{m-1})$  the dimension  $\dim (Vd_{k;I}^\A)$ depends periodically on $i_{m-1}$ for all $i_{m-1}$ sufficiently large.  
\end{conj}

Notice that  Conjecture~\ref{conj:period} follows from Conjecture~\ref{conj:CKW} in the special case $k=3,\; m=5$. Unfortunately, we do not have a complete description of the length of this period in terms of the fixed part of $I$  and it might be quite tricky. 

For the $\mathcal{BC}_k$-localized variety $Vd_{k;I}^{BC}$ we have, except for $k=3$,  only conjectures, supported
by many calculations. 

\begin{conj}\label{th:k*k+1??} For any integer sequence $I=\{0=i_0<i_1<i_2<...<i_k\}$ of length $k+1$ with $\gcd(i_1,...,i_k)=1$ the following facts are valid. 
\begin{itemize}

\item[(i)] ${\text codim} (Vd_{k;I}^{BC})=2$;

\item[(ii)] The quotient ring $\RR_{I}^{BC}$ is Cohen-Macaulay;

\item[(iii)] There is a $\C[x_1,\ldots,x_n]=R$-resolution of $\RR_{I}^{BC}$ of the form
$$0\rightarrow\oplus_{j=0}^{k-1} R[-N+j+\binom{k}{2}]\rightarrow\oplus_{j=1}^kR[-N+i_j+\binom{k}{2}]\oplus R[-N+ki_1]\rightarrow R\rightarrow R\rightarrow\RR_{I}^{BC}\rightarrow 0$$

\item[(iv)] The Hilbert series $Hilb_I^{BC}(t)$ of $\RR_{I}^{BC}$ is given by the formula
$$Hilb_I^{BC}(t)=\left(1-\sum_{j=1}^kt^{N-j-\binom{k}{2}}-t^{N-ki_1}+\sum_{j=1}^{k-1}t^{N-i_1-j}-\binom{k}{2}\right)/(1-t)^k$$
where $N=\sum_{j=1}^k i_j$; 

\item[(v)] $\deg (Vd_{k;I}^{BC})=\sum_{1\le j<l\le k}i_ji_l
-{k\choose2}\sum_{j=1}^ki_j+{k+1\choose3}(3k-2)/4-\binom{k}{2}i_1(i_1-1)$;

\item[(vi)]  The ideal $\mathcal I_{I}^{BC}$  is always generated by  $k$ generators. It is generated by two elements  (i.e., is a complete intersection) if $i_1\le k-1$. 

\end{itemize}
\end{conj}
\begin{conj}\label{Brel}
Let $S_k,\ldots,S_1$ be as in Theorem~ \ref{th:general} and $G_0=s_{i_k-i_1-k+1,\ldots,i_2-i_1-1}$. Then, for $s=0,\ldots,k-1$ we have
$$h_{i_k-i_1-s}S_k-h_{i_{k-1}-i_1-s}S_{k-1}+\cdots+(-1)^{k-2}h_{i_2-i_1-s}S_2+(-1)^{k-1}h_{-s}S_1+$$
$$(-1)^ks_{i_1-1,\ldots,(i_1-1)^{k-1},k-1-s}G_0.$$
Here $h_i=0$ if $i<0$ and $h_{i,\ldots,i,j}=0$ if $j>i$, and $(i_1-1)^{k-1}$ means $i_1-1,\ldots,i_1-1$ ($k-1$ times).
\end{conj}
That the ring is CM follows from the fact that the ideal is generated by the maximal minors of a $t\times m$-matrix in the ring of Laurent polynomials. To prove the theorem it suffices to prove the relations between the Schur polynomials. Unfortunately we have managed to do that only for $k=3$.

\section{Final remarks}
\medskip
Here we briefly explain the source of our interest in Vandermonde varieties. 
In 1977 J.~H.~Loxton and A.~J.~van der Poorten formulated an important  conjecture (Conjecture $1^\prime$ of \cite{LoPo}) claiming  that there exists a constant $\mu_k$ such that any integer recurrence of order $k$  either has at most $\mu_k$ integer zeros or has infinitely many zeros. 

This conjecture was first settled by W.~M.~Schmidt in 1999, see \cite {Sch} and also by J.~H.~Evertse and H.~P.~Schlickewei, see \cite{EvSl}. 

The upper bound for $\mu_k$ obtained in \cite{Sch} was 
$$\mu_k<e^{e^{e^{3k \log k}}},$$
which was later improved by the same author to 
$$\mu_k <e^{e^{e^{20 k}}}.$$

Apparently the best known at the moment upper bound for $\mu_k$ was obtained in \cite{All} and is given by 
$$\mu_k<e^{e^{k^{\sqrt {11 k}}}}.$$

Although the known upper bounds are at least double exponential it seems plausible that the realistic upper bounds should be polynomial. The only known nontrivial lower bound for $\mu_k$ was found in \cite{BaBe}  and is given by 
$$\mu_k\ge \binom{k+1} {2} -1.$$
One should also mention the non-trivial exact result of F.~Beukers showing that for sequences of {rational} numbers obtained from recurrence relations of length $3$ one has $\mu_3=6$, see \cite {Beu}. 

The initial idea of this project was to try to obtain upper/lower bounds for $\mu_k$ by studying  algebraic and geometric properties of Vandermonde varieties but they seems to be quite complicated.    
Let us  finish with some further problems and comments on them, that we got with an extensive computer search.
Many questions related to the Skolem-Mahler-Lech theorem translate immediately into questions about $V_{k;I}$. For example, one can name the following formidable challenges. 

\medskip
\noindent
{\em Problem 3.} For which pairs $(k;I)$ the variety $V_{k;I}$ is empty/non-empty? More generally, what is the dimension of $V_{k;I}$?

\medskip
We made a complete computer search for $\RR_I^\A$ and some variants where we removed solutions on the coordinate planes and axes, and looked for
arithmetic sequences, for $(0,i_1,i_2,i_3)$, $0<i_1<i_2<i_3$, $i_3\le13$ (so $k=3$, $m=4$). The only cases when 
$V_{k;I}$ was empty were $I=(0,1,3,7)$ and $I=(0,1,3,9)$ and their "duals" $(0,4,6,7)$ and $(0,6,8,9)$.  We suspect that our exceptions are the only possible. 
For $k=3$, $m=5$ we investigated $I=(0,i_1,i_2,i_3,i_4)$, $0<i_1<i_2<i_3<i_4$, $i_4\le9$. For $i_1=1$ about half of the cases had the expected dimension. For $(k,m)=(3,6)$, $i_5\le10$, for $(k,m)=(4,6)$, $i_5\le9$ and for $(k,m)=(5,8)$, $i_7\le10$, most cases were of expected dimension.
The corresponding calculations for  $\RR_{I}^{BC}$, $(k,m)=(3,5)$, $i_4\le9$, showed that about half of the cases had expected codimension.

\medskip
\noindent
{\em Problem 4.} For which pairs $(k;I)$ any solution of a linear recurrence vanishing at $I$ must have an additional integer root  outside $I$? More specifically, for which pairs $(k;I)$ any solution of a linear recurrence vanishing at $I$ must vanish infinitely many times in $\bZ$? In other words,  for which pairs $(k;I)$ the set of all  integer zeros of the corresponding  solution of any recurrence relation from $V_{k;I}$  must necessarily contain an arithmetic progression?

\medskip
For example, in case $k=3,\;m=4$ we found that the first situation occurs for $4$-tuples $(0,1,4,6)$ and $(0,1,4,13)$ which both force a non-trivial solution of a third order recurrence vanishing at them to vanish at the $6$-tuple $(0,1,4,6,13,52)$, which is the basic example in \cite {Beu}. The second  situation occurs if in a $4$-tuple $I=\{0,i_1,i_2,i_3\}$ two differences between its entries coincide, see \cite {Beu}. But this condition is only sufficient  and no systematic information is available.  Notice that for any pair $(k;I)$ the variety $\overline V_{k;I}$ is weighted-homogeneous where the coordinate $\al_i,\;i=1,\ldots,k$ has   weight $i$. (This action corresponds to the scaling of the characteristic roots of \eqref{eq:Char}.) 

\smallskip
 We looked for cases containing  an arithmetic sequence with difference at most 10 and we found cases which gave arithmetic sequences with difference
2,3,4, and 5, and a few cases which didn't give any arithmetic sequences.

\medskip
\noindent
{\em Problem 5.} Is it true that if an $(k+1)$-tuple $I$ consists of  two pieces of arithmetic progression with the same difference then any exponential polynomial vanishing at $I$ contains an arithmetic progression of integer zeros? 

\medskip
\noindent
{\em Problem 6.} 
If the answer to the previous question is positive is it true that there are only finitely many exceptions from this rule leading to only arithmetic progressions? 

\medskip
\noindent 
{Finally a problem similar to that of J.~H.~Loxton and A.~J.~van der Poorten can be formulated for real zeros of exponential polynomials instead of integer. Namely, the following simple lemma is true.  

\begin{lm}\label{lm:real} Let $\la_1,...,\la_n$ be a arbitrary finite set of (complex) exponents having all distinct real parts then an arbitrary exponential polynomial of the form $c_1e^{\la_1z}+c_2e^{\la_2z}+..+c_ne^{\la_nz},\; c_i\in \bC$ has at most finitely many real zeros. 
\end{lm} 

\noindent
{\em Problem 7.} Does there  exist an upper bound on the maximal number real for the set of exponential polynomials given in the latter lemma in terms of $n$ only?

\medskip
\noindent 
{\em Problem 8.} What about non-regular cases? Describe their relation to the existence of additional integer zeros and arithmetic progressions as well as additional Schur polynomials in the ideals.


\begin{thebibliography}{Dillo 83}

\bibitem{All} P.~B.~Allen, {On the multiplicity of linear recurrence sequences}, J. Number Th., vol. 126 (2007), 212--216.  

\bibitem{BaBe} E.~Bavencoffe ad J.~P.~B\'ezevin, {Une famille Remarkable de Suites Recurrentes Lineares}, Monatsh. Math., vol. 120 (1995), 189--203.

\bibitem{Beu} F. Beukers, {The zero-multiplicity of ternary recurrences}, Compositio Math. 77 (1991), 165-177. 

\bibitem{CKW} A.~Conca, C.~Krattenthaler, J.~Watanabe,   {Regular sequences of symmetric polynomials}, Rend. Sem. Mat. Univ. Padova 121 (2009), 179--199.

\bibitem{DZ} R.~Dvornicich,  U.~Zannier,  {Newton functions generating symmetric fields
and irreducibility of Schur polynomials},  Adv. Math.  222  (2009),  no. 6,
1982--2003.

\bibitem{EN} J.~A.~Eagon,  D.~G.~Northcott,  {Ideals defined by matrices and a certain
complex associated with them},  Proc. Roy. Soc. Ser. A  269  (1962) 188--204.

\bibitem{EvSl} J.~H.~Evertse ad H.~P.~Schlickewei, {A qualitative version of the Absolute Subspace theorem,} J. reine angew. Math. vol. 548 (2002), 21--127.  

\bibitem{Fu} W.~Fulton, {Flags, Schubert polynomials, degeneracy loci, and determinantal formulas,} Duke Math. J., vol 65 (3) (1991),  381--420. 

\bibitem{Le} C.~Lech, {A note on Recurring Series}, Ark. Mat. 2,  (1953), 417--421.

\bibitem{LoPo} J.~H.~Loxton and A.~J.~van der Poorten, {On the growth of recurrence sequences}, Math. Proc. Camb. Phil. Soc. vol. 81 (1977), 369--377.  


\bibitem{Ma} K.~Mahler, {Eine arithmetische Eigenschaft der Taylor-Koeffizienten rationaler Funktionen}, Proc. Kon. Nederl. 
Akad. Wetensch. Amsterdam, Proc. 38 (1935), 50--60.

\bibitem{Mu}D.~Mumford: The red book of varieties and schemes. Second, expanded edition. Includes the Michigan lectures (1974) on curves and their Jacobians. With contributions by Enrico Arbarello. Lecture Notes in Mathematics, 1358. Springer-Verlag, Berlin, 1999. x+306 pp. 

\bibitem {Sk} Th.~Skolem, {Einige S\"atze \"uber gewisse Reihenentwicklugen und exponentiale Beziehungen
mit Anwendung auf diophantische Gleichungen}, Oslo Vid. akad. Skrifter {\bf I} 1933 Nr. 6.

\bibitem {Sch} W.~M.~Schmidt, {The zero multiplicity of linear recurrence sequences,} Acta Math. vol. 182 (1999), 243--282.

\bibitem{St}
R.~Stanley, {\em Enumerative combinatorics}. Vol. I. With a foreword by Gian-Carlo Rota. The Wadsworth and Brooks/Cole Mathematics Series. Wadsworth and Brooks/Cole Advanced Books and Software, Monterey, CA, 1986. xiv+306 pp.  

\bibitem{Tam} H.~Tamvakis, {The theory of Schur polynomials revisited}, arXiv:1008.3094v1. 

\end{thebibliography}
\end{document}